\documentclass{ifacconf}

\usepackage{graphicx}      
\usepackage{natbib}        
\usepackage{amsmath}
\usepackage{amssymb}
\usepackage{amsfonts}
\usepackage{tikz}
\usetikzlibrary{calc}
\usepackage{stfloats}
\usepackage{pgfplots}
\pgfplotsset{compat=1.18}
\usepackage{epstopdf}
\usepackage{enumerate}

\definecolor{myblue}{RGB}{31,119,180}
\definecolor{myorange}{RGB}{255,127,14}
\definecolor{mygreen}{RGB}{44,160,44}
\definecolor{myred}{RGB}{214,39,40}
\definecolor{mypurple}{RGB}{148,103,189}
\definecolor{mycyan}{RGB}{23,190,207} 
\definecolor{myteal}{RGB}{0, 128, 128}

\newcommand{\supobsv}{\bar{\mathcal{O}}}
\newcommand{\subobsv}{\underline{\mathcal{O}}}
\newcommand{\bfA}{\mathbf{A}}
\newcommand{\bfB}{\mathbf{B}}
\newcommand{\bfC}{\mathbf{C}}
\newcommand{\bfCs}{{\bfC^*}}
\newcommand{\bfE}{\mathbf{E}}
\newcommand{\bfEt}{\bfE(t)}
\newcommand{\bfK}{\mathbf{K}}
\newcommand{\bfL}{\mathbf{L}}

\newcommand{\mcX}{\mathcal{X}}
\newcommand{\T}{^T}

\newcommand{\tz}{\tilde{z}}
\newcommand{\bT}{\bar{T}}
\newcommand{\etx}{e_{\tilde{x}}}
\newcommand{\bd}{\bar{d}}
\newcommand{\bwp}{{\bar{w}'}}

\begin{document}
\begin{frontmatter}


\title{Secure estimator design for \\ Lur'e-type systems with nonuniformly and synchronously sampled measurements under attacks [extended version]\thanksref{footnoteinfo}}

\thanks[footnoteinfo]{This work is funded through the CETPartnership’s ProRES project under the Joint Call 2024, co‑funded by the European Commission (Grant Agreement No. 101069750) and the Dutch Research Council NWO (File No. EP.1602.24.001).
}

\author[First]{Julian Gootzen} 
\author[First]{Michelle S. Chong}

\address[First]{Department of Mechanical Engineering, \\ Eindhoven University of Technology, 
   the Netherlands \\ (e-mail: j.s.j.gootzen@student.tue.nl, m.s.t.chong@tue.nl).}

\begin{abstract}                
Motivated by the need for real-time health monitoring of power distribution grids, we propose a secure state estimator design for continuous time Lur’e type systems with non-uniformly and synchronously sampled outputs which have potentially been maliciously corrupted. The secure state estimator provides state estimates with accuracy independent of the sensor attack, when less than half of the sensors are under attack and when all inter-sample times are upper bounded. We show convergence of the state estimation error under an impulsive system framework and provide an upper bound on the estimation error that is independent of the attack signals. The stability conditions are formulated as linear matrix inequalities, which can be used to design the observer parameters.  We demonstrate the capabilities of the proposed secure state estimator on a low-voltage power distribution grid.
\end{abstract}


\end{frontmatter}

\section{Introduction}
Secure state estimation refers to the problem of estimating the state of a plant when the sensors of the plant are being manipulated with malicious intent. In particular, the state estimation error has an upper bound that is independent of the attack on the sensors. Cyber-physical systems are particularly vulnerable to sensor attacks as the communication medium (cyber component) adds a vulnerable entry point for malicious actors, from where measurements can be manipulated. An example of such a system is the low-voltage electricity distribution grid, which motivated this work and is used as a case study in this paper. Each customer sends voltage measurements to a centralized monitoring center, where an estimate of the distribution grid's state is generated. This transmission of measurements from the customer to the monitoring center realizes such a vulnerability that can be exploited.

In recent years, several designs of secure state estimators for nonlinear systems have been proposed, such as in \cite{Chong2020AAttacks}, \cite{Kim2019SensorAttacks}, and \cite{Yang2022SensorFaultTolerant}. However, to the best of our knowledge, no work addresses the case where potentially manipulated measurements are only available at sampling times.

State estimation of nonlinear systems \textit{in the absence of sensor attacks}, with non-uniformly sampled measurements has been studied before. In \cite{raff2008observer}, an observer is proposed which applies a zero order hold in between samples. \cite{DING2009324} and \cite{ARCAK20041931} implement a discrete-time observer of a discretized version of the plant.

In this paper, we propose a secure state estimator for sector bounded Lur'e-type systems with nonuniformly and synchronously sampled measurements. We build on the multi-observer as proposed in \cite{Chong2020AAttacks}, which uses continuously available measurements. The multi-observer provides a secure state estimate when more than half of the outputs remain attack-free. Since we no longer have continuous access to the measurements, we implement a zero order hold scheme in between samples. We aim to show that such an observer is input-to-state stable with respect to noise, disturbances and sampling errors under an impulsive system framework, using a finite-dimensional Lyapunov function as proposed in \cite{nagh2008exponential} and \cite{raff2008observer}. The stability conditions that result from this analysis are formulated as Linear Matrix Inequalities (LMIs), which are used to design observer matrices.

We demonstrate the efficacy of the proposed observer design by observing a low-voltage power distribution network with several customers. The observer receives voltage measurements from the customers at nonuniform sampling times synchronously, while the measurements strictly less than half of the measurements are being maliciously manipulated. Simulations confirm that the proposed secure observer design is able to provide an accurate state estimate in spite of the attacks.

The main contributions of this paper are:
\begin{itemize}
    \item an extension of the multi-observer based secure state estimator  in \cite{Chong2020AAttacks} to accomodate nonuniformly but synchronously sampled measurement data;
    \item an LMI-based observer design;
    \item the application of the secure state estimator to the monitoring of power distribution networks.
\end{itemize}

\section{Notation} \label{sec:notation}
Let $\mathbb{R}=(-\infty,\infty)$, $\mathbb{R}_{\geq0} = [0,\infty)$, $\mathbb{R}_{>0} = (0,\infty)$, $\mathbb{N}=\{0,1,2,3,\dots\}$, $[k] := \{1,2,\dots,k\}$. The cardinality of a set $S$ is denoted by $|S|$. The identity matrix of dimension $n$ is denoted by $\mathbb{I}_n$. Given $u\in \mathbb{R}^{n_u}$ and $v \in \mathbb{R}^{n_v}$, the column vector $(u^T,v^T)^T$ is denoted by $(u,v)$ and the inner product is denoted by $\langle u,v \rangle$ (for $n_u=n_v)$. Let $\operatorname{diag}(x_1,x_2,\dots,x_k)$ denote the diagonal matrix with the elements of the vector $(x_1, x_2, \cdots, x_k)$ placed along the diagonal of the matrix. Let $\lambda_{min}(A)$ and $\lambda_{max}(A)$ denote the smallest and largest eigenvalue of matrix $A$. The Euclidean norm of a matrix is denoted $|A|=\sqrt{\lambda_{max}(A^TA)}$. The euclidean norm of a vector $x$ is denoted $|x|=\sqrt{x^Tx}$. Let $||\eta||_{[t_0,t_1]}$ denote the sup-norm $\sup_{s\in [t_0,t_1]} |\eta(s)|$. Let $\textbf{1}_n$ denote a vector in $\mathbb{R}^n$ with all elements having the value 1. The Kronecker product of two matrices $A=[a_{ij}] \in \mathbb{R}^{m \times n}$ and $B$ is denoted by $A \otimes B = \begin{bmatrix} a_{11}B & \dots & a_{1n}B \\ \vdots & \ddots & \vdots \\ a_{m1}B & \dots & a_{mn}B \end{bmatrix}$. Standard definitions for the class $\mathcal{K_{\infty}}$ and $\mathcal{KL}$ comparisons are used, see \cite{hybrid-dynamical-systems}.

\section{Problem statement}\label{sec:problem-statement}
We aim to provide a continuous state estimate of a power distribution network with $N_c$ customers interconnected in a radial fashion. The voltage is measured at each customer and transmitted in packets synchronously over a shared communication channel. This communication channel is vulnerable and hence susceptible to attacks. The power distribution system is a Lur'e type system of the following form:
\begin{equation}
    \begin{split}
        \dot{x}(t) &= Ax(t) + B\phi(m), \; m(t) = Cx(t) + u(t) + d(t), \\
        y(t) &= m(t) + a(t) + \eta(t),
    \end{split}
    \label{eq:standard-system-form}
\end{equation}
where $x\in\mathbb{R}^{N_c}$ is the state vector;  $u\in\mathbb{R}^{N_c}$ is a measured input vector; $a = (a_1, a_2, \dots, a_{N_c})$ is the attack vector, where $a_i \ : \ \mathbb{R}_{\geq 0 } \to \mathbb{R}, \ i \in [N_c]$ is a potentially unbounded, unknown attack signal on sensor $i$; $d \ : \ \mathbb{R}_{\geq0} \to  \mathbb{R}^{N_c}$ and $\eta \ : \ \mathbb{R}_{\geq0} \to \mathbb{R}^{N_c}$ are unknown process disturbance and measurement noise vectors respectively; $A \in \mathbb{R}^{N_c \times N_c}$ is the state matrix; $B\in \mathbb{R}^{N_c \times N_c}$ is the input matrix; $C \in \mathbb{R}^{N_c \times N_c}$ is the output matrix; and the nonlinearity $\phi: \mathbb{R}^{N_c} \to \mathbb{R}^{N_c}$ satisfies the following:
\begin{assum}\label{as:nonlinearity}
    Each component of the nonlinearity,
    \begin{equation}
        \phi(m)= \left(\phi_1(m_1), \phi_2(m_2), \dots, \phi_{N_c}(m_{N_c})\right),
    \end{equation}
    $\phi_i \ : \mathbb{R} \to \mathbb{R}$ satisfies the incremental sector condition
    \begin{equation}
        0 \leq \frac{\phi_i(a)-\phi_i(b)}{a-b} \leq \zeta_i, \quad \forall\ a,b \in \mathbb{R}, \ a \neq b,
        \label{eq:incremental-sector-condition}
    \end{equation}
    where $\zeta_i > 0$ for all $i\in[N_c]$.
\end{assum}
The plant as in \eqref{eq:standard-system-form} is represented in Figure \ref{fig:plant-network} by the teal layer \textcolor{myteal!40}{$\blacksquare$}. We provide further details on the model of the low voltage power distribution network in Section \ref{sec:power-distirbution-network}. Let us now make the following non-conservative assumptions about the sensor attack.
\begin{assum}
    The attack vector $a:=(a_1,a_2,\dots,a_{N_c})$ satisfies the following:
    \begin{enumerate}[(i)]
        \item sensors $i \in [N_c]$ with $N_c\in\mathbb{N}$ which are not under attack satisfy $a_i(t)=0$ for all $t \geq 0$;
        \item the set of attacked sensors is constant;
        \item no more than $N_a\in\mathbb{N}$ sensors are under attack, where $2N_a < N_c$.
    \end{enumerate}
    \label{as:attack}
\end{assum}
Assumption \ref{as:attack}-(iii) is crucial as our estimation algorithm employs a majority voting-like scheme such that the sensors uncorrupted by attacks are used in the generation of the state estimate. This requirement is also employed in the secure state estimation schemes of \cite{Chong2020AAttacks, van2024resilient}.

Contrary to existing works, the measurement data $m_i$ from each sensor $i\in \mathbb{N}$ is transmitted over a communication channel in packets at times $t_k,\ k \in \mathbb{N}$, with non-uniform sampling intervals, i.e., the inter-sample times $t_{k+1}-t_k$ vary. We assume that transmission delays are negligible and hence, the received packets from each $i$-th sensor is $y_i(t_k)$ as modelled in system \eqref{eq:standard-system-form}, which incorporates disturbances and potential malicious corruption due to the vulnerability of the communication channel. We define the set of all sampling times as $\mathcal{D} = \{t_k\ : \ k\in\mathbb{N} \}$ and let $\mathcal{C}$ denote all times $t$ when the outputs are not sampled: $\mathcal{C} = \mathbb{R}_{\geq 0} \setminus \mathcal{D}$. The network is schematically represented in Figure \ref{fig:plant-network} as the red layer \textcolor{myred!40}{$\blacksquare$}. We now formalize the assumptions on the sensor sampling. 
\begin{assum}
    The set of sampling times $\mathcal{D}$ is an increasing sequence that tends to infinity. Let there also exist $(\underline{T},\bar{T}) \in \mathbb{R}_{\geq 0}^2$, with $\underline{T} < \bar{T}$ such that
    \begin{equation}
        t_0 = 0, \quad t_{k+1} - t_k \in [\underline{T},\bar{T}], \quad k \in \mathbb{N}.
    \end{equation}
    \label{as:sampling}
\end{assum}
\vspace{-13pt}
Note that each customer transmitting their measurement data synchronously is not necessarily a realistic assumption for a real low voltage electricity distribution grid.

\begin{figure}[h]
    \centering
    \begin{tikzpicture}[
            node distance=0.5cm,
            block/.style={rectangle, draw=black!100, thick, minimum size=5mm, text centered},
            vertline_node/.style={
                draw=none,
                rotate=0,
                anchor=center,
                inner sep=2pt,
                text centered,
                append after command={
                    \pgfextra{
                        \draw[black, very thick] 
                            ([xshift=0\pgflinewidth] \tikzlastnode.south west) -- 
                            ([xshift=0\pgflinewidth] \tikzlastnode.north west);
                        \draw[black, very thick] 
                            ([xshift=0\pgflinewidth] \tikzlastnode.south east) -- 
                            ([xshift=0\pgflinewidth] \tikzlastnode.north east);
                    }
                }
            },
            legend/.style={rectangle, draw=black, minimum size = 3mm}
        ]
        \filldraw[fill=myteal!20] (-3,2.1) -- (3,2.1) -- (3,0.40) -- (-3,0.40) -- (-3,2.1);
        \node (plant) at (0,1.7) [block, text width = 2cm, fill=myteal!40] {Plant};
        \node (c1) at (-2,0.9) [block, text width = 1.1cm, fill=myteal!40] {\small Cust. $1$};
        \node (c2) at (-0.5,0.9) [block, text width = 1.1cm, fill=myteal!40] {\small Cust. $2$};
        \node (cdots) at (0.78,0.9) {$\cdots$};
        \node (cN) at (2,0.9) [block, text width = 1.2cm, fill=myteal!40] {\small Cust. $N_c$};

        \node (comm) at (0,-0.4) [block, fill=myred!40, text width = 6cm, minimum height = 0.6cm]   {\footnotesize Communication Channel};

        \draw[->] (c1) -- (-2,-0.1) node[right, pos=0.65] {\small $m_1(t)$};
        \draw[->] (c2) -- (-0.5,-0.1) node[right, pos=0.65] {\small $m_2(t)$};
        \draw[->] (cN) -- (2,-0.1) node[right, pos=0.6] {\small $m_{N_c}(t)$};

        \draw[->] (-2,-0.7) -- (-2,-1.3) node[right, pos=0.5] {\small $y_1(t_k)$};
        \draw[->] (-0.5,-0.7) -- (-0.5,-1.3) node[right, pos=0.5] {\small $y_2(t_k)$};
        \draw[->] (2,-0.7) -- (2,-1.3) node[right, pos=0.5] {\small $y_{N_c}(t_k)$};

        \draw[->] (-3.5, -0.4) -- (comm.west) node[above, pos=0] {\small $\eta(t)$};
        \draw[->] (3.5,-0.4) -- (comm.east) node[above, pos=0] {\small $a(t)$};
    \end{tikzpicture}
    \vspace{-1em}
    \caption{A schematic diagram of the plant and communication network}
    \label{fig:plant-network}
\end{figure}
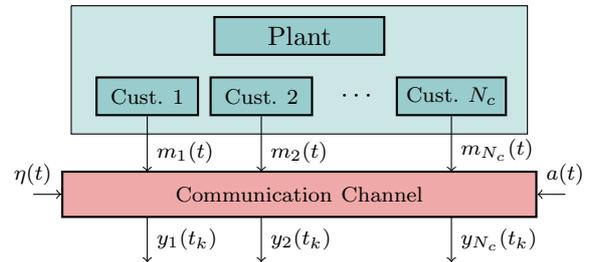
We now aim to create a \textit{secure state estimate} $\hat{x}$ of the state $x$ in system \eqref{eq:standard-system-form}, satisfying Assumption \ref{as:nonlinearity}, whose sensors have been compromised by attack $a$ satisfying Assumption \ref{as:attack} and where sensor readings $y$ are only available at sample times $t_k$, satisfying Assumption \ref{as:sampling}. A state estimator is referred to as \textit{secure} if the bound on the state estimation error $x-\hat{x}$ is independent of any attack vector $a(t)$, satisfying Assumption \ref{as:attack}, i.e., for all initial conditions $(x(0),\hat{x}(0)) \in \mathbb{R}^{N_c} \times \mathbb{R}^{N_c}$ and $t \geq 0$, we have
\begin{equation*}
    |x(t)-\hat{x}(t)| \leq \max\{ \beta(|x(0) - \hat{x}(0)|,t), \gamma(||(\delta,d,\eta)||_{[0,t]})\},
\end{equation*}
where $\delta(s):=y(s)-y(t_k)$ is the sampling induced disturbance, $\beta \in \mathcal{KL}$ and $\gamma \in \mathcal{K_\infty}$.


\section{Secure state estimator design}\label{sec:observer-design}
In order to mitigate the sensor attacks $a(t)$, we implement a multi-observer setup as in \cite{Chong2020AAttacks}, which has been shown to be a secure state estimator when provided with continuous measurements $y(t)$. However, due to the aforementioned communication constraints, these measurements are now sampled non-uniformly, as described in Assumption \ref{as:sampling}. Hence, a zero-order hold is applied in between sampling times. We now introduce the multi-observer and selection procedure, and conclude with an upper bound for the sampled-data estimation error system. 

\subsection{Multi-observer}
The multi-observer has \textit{super}- and \textit{sub-observers}, which are denoted by $\supobsv$ and $\subobsv$ and are represented by the orange \textcolor{myorange!40}{$\blacksquare$} and green \textcolor{mygreen!40}{$\blacksquare$} blocks in Figure \ref{fig:observer-diagram} respectively. All observers in the multi-observer use a subset of the sampled outputs $y(t_k)$, given by $\supobsv = \{ {S} \subseteq [N_c]: |{S}| = N_c - N_a  \}$ and $\subobsv =\{ {S} \subseteq [N_c]: |{S}| = N_c - 2N_a  \}$. This results in $N_{\mathcal{O}} = \bar{o} +\underline{o}$ observers in the multi-observer. With $\bar{o} = \binom{N_c}{N_c-N_a}$ and $\underline{o} = \binom{N_c}{N_c-2N_a}$, where in turn $\binom{n}{k}$ denotes the number of $k$-element subsets of an $n$-element set. Let $\supobsv_{i},\ i \in [\bar{o}]$ or $\subobsv_{j},\ j \in [\underline{o}]$ denote the set of sensor indices used by a single super- or sub-observer, which generate the state estimates $\hat{\bar{x}}^i$ and $\hat{\underline{x}}^j$ respectively.

In order to make statements about the state estimates of \textit{both} the super- and sub-observers compactly, we define $\hat{x}^{S},\ S \in \{\supobsv,\subobsv\}$ as the state estimate generated by the observer using sensor set $S$, i.e. using the outputs ${y}_i$ for $i\in S$. For example, consider the observer using outputs $S=\{1,3,4\}$. The state estimate is denoted $\hat{x}^{\{1,3,4\}}$, where the sensor indices in the superscript denote that this state estimate is generated by the observer using the outputs $y_{i}$ for $i\in S$. The upper or lower bars are omitted when referring to either super- or sub-observers generally. Furthermore, we will denote the output this observer uses as ${y}_{\{1,3,4\}} = ({y}_1,{y}_3,{y}_4)$, where the sensor indices in the subscript denote the selection of those entries from the full output vector $y=(y_1,y_2,\dots,y_{N_c})$.

The state estimates of all super- and sub-observers $S\in \{\supobsv,\subobsv\}$ are generated by 
\begin{equation}
    \begin{split}
        \dot{\hat{x}}^S &= A\hat{x}^S + B\phi(\hat{m}^S) + L^S(C_S\hat{x}^S(t_k) + u_S(t_k) - {y}_S(t_k)) \\
        \hat{m}^S &= C\hat{x}^S + u + K^{S}(C_{S}\hat{x}^S(t_k) + u_{S}(t_k) - {y}_{S}(t_k))
    \end{split}
    \label{eq:multi-observer}
\end{equation}
where $K^{S} \in\mathbb{R}^{N_c \times |S|}$ and $L^{S} \in\mathbb{R}^{N_c \times |S|}$ are observer gains to be designed. Note that $C_{S}$ denotes the submatrix of $C$ created by stacking the rows indexed by $s \in S$ and $u_{S}$ denotes a vector consisting of the components of the vector $u$ indexed by $s \in S$.

\subsection{State estimate selection procedure}
We now introduce the selection procedure that is used to select the state estimate $\hat{x}$ from all state estimates generated by the super-observers $\hat{\bar{x}}^i,\ i \in [\bar{o}]$. We select the super-observer state estimate $\hat{\bar{x}}^i$ that is \textit{most consistent} with all its sub-observers. We call $\subobsv_j$ a sub-observer \textit{of} $\supobsv_i$ if it only uses outputs that are also used by $\supobsv_i$, i.e. $\subobsv_j \subseteq \supobsv_i, \ j \in [\underline{o}], \ i \in [\bar{o}]$. A consistency mapping $\Phi: \mathbb{R}^{N_cN_{\mathcal{O}}} \rightarrow \mathbb{R}^{N_c}$ selects the state estimate of the most consistent super-observer. Let $\pi^i \in \mathbb{R}_{\geq 0}$ denote a consistency measure corresponding to the state estimate $\hat{\bar{x}}^i,\ i \in [\bar{o}]$. It quantifies the consistency of the super-observer state estimates with respect to the state estimates constructed by their sub-observers $\hat{\underline{x}}^j,\ j \in [\underline{o}]$,
\begin{equation}\label{eq:consistency-measure}
    \pi^{i}(t) = \max_{\{j \in [\underline{o}] \ : \ \subobsv_j \subset \supobsv_i \}} |\hat{\bar{x}}^{i}(t) - \hat{\underline{x}}^{j}(t)|, \ i \in [\bar{o}],
\end{equation}
for $t \geq 0$. The state estimate $\hat{x}(t)$ is selected from all state estimates generated by the super-observers as
\begin{equation}\label{eq:selection-criterion}
    \hat{x}(t)=\hat{\bar{x}}^{\sigma(t)}(t), \quad \sigma(t) = \operatornamewithlimits{arg \ min}_{i} \pi^{i}(t), \ i \in [\bar{o}].
\end{equation}
The complete selection procedure is denoted by $\hat{x} = \Phi(\hat{z})$, where $\hat{z}=(\hat{\bar{x}}^1,\dots,\hat{\bar{x}}^{\bar{o}},\hat{\underline{x}}^1,\dots,\hat{\underline{x}}_{\underline{o}})$, and schematically represented by the cyan layer in Figure \ref{fig:observer-diagram} \textcolor{mycyan!20}{$\blacksquare$}.
\begin{figure}[ht]
    \centering
    \begin{tikzpicture}[
            node distance=0.5cm,
            block/.style={rectangle, draw=black!100, thick, minimum size=5mm, text centered},
            vertline_node/.style={
                draw=none,
                rotate=90,
                anchor=center,
                inner sep=2pt,
                text centered,
                append after command={
                    \pgfextra{
                        \draw[black, thick] 
                            ([xshift=0.5\pgflinewidth] \tikzlastnode.south west) -- 
                            ([xshift=0.5\pgflinewidth] \tikzlastnode.south east);
                        \draw[black, thick] 
                            ([xshift=-0.5\pgflinewidth] \tikzlastnode.north west) -- 
                            ([xshift=-0.5\pgflinewidth] \tikzlastnode.north east);
                    }
                }
            },
            legend/.style={rectangle, draw=black, minimum size = 3mm}
        ]
        
        \filldraw[fill=myteal!20] (-3,-2.4) -- (3,-2.4) -- (3,-3);
        \filldraw[fill=myred!20] (3,-3) -- (-3,-3) -- (-3,-2.4);
        \node (PN) at (0,-2.7) {Plant \& Comm. Network};
        
        \draw[->] (-2,-3) -- (-2,-3.7) node[right, pos=0.5] {\small $y_1(t_k)$};
        \draw[->] (-0.5,-3) -- (-0.5,-3.7) node[right, pos=0.5] {\small $y_2(t_k)$};
        \draw[->] (2,-3) -- (2,-3.7) node[right, pos=0.5] {\small $y_{N_c}(t_k)$};
        \draw[-, very thick] (-3,-3.7) -- (3,-3.7);
    
        \filldraw[fill=mypurple!20] (-3.4,-4.4) -- (3.4,-4.4) -- (3.4,-6.6) -- (-3.4,-6.6) -- (-3.4,-4.4);

        \filldraw[fill=myorange!20] (-3,-4.6) -- (-0.1,-4.6) -- (-0.1,-5.4) -- (-3,-5.4) -- (-3,-4.6);
        \filldraw[fill=mygreen!20] (0.1,-4.6) -- (3,-4.6) -- (3,-5.4) -- (0.1,-5.4) -- (0.1,-4.6);
        
        \node (osup1) at (-2.4,-5) [block, fill=myorange!40] {$\bar{\mathcal{O}}_{1}$};
        \node (osupdots) at (-1.55,-5) {$\cdots$};
        \node (osupo) at (-0.7,-5) [block, fill=myorange!40] {$\bar{\mathcal{O}}_{\bar{o}}$};

        \draw[->, very thick] (-2.4,-3.7) -- (osup1.north) node[right, pos=0.4] {\small ${y}_{\bar{\mathcal{O}}_1}(t)$};
        \draw[->, very thick] (-0.7,-3.7) -- (osupo.north) node[right, pos=0.4] {\small ${y}_{\bar{\mathcal{O}}_{\bar{o}}}(t)$};

        \node (osub1) at (0.7,-5) [block, fill=mygreen!40] {$\underline{\mathcal{O}}_{1}$};
        \node (osubdots) at (1.55, -5) {$\cdots$};
        \node (osubo) at (2.4,-5) [block, fill=mygreen!40] {$\underline{\mathcal{O}}_{\underline{o}}$};

        \draw[->, very thick] (0.7,-3.7) -- (osub1.north) node[right, pos=0.4] {\small ${y}_{\underline{\mathcal{O}}_1}(t)$};
        \draw[->, very thick] (2.4,-3.7) -- (osubo.north) node[right, pos=0.4] {\small ${y}_{\underline{\mathcal{O}}_{\underline{o}}}(t)$};

        \node (phi) at (0,-6.2) [block, text width=5.8cm, fill=mycyan!20, line width=0.4] {$\Phi(\hat{z})$};
        
        \draw[-] (phi.south) -- (0,-6.8);
        \draw[->] (0,-6.8) -- (0,-6.9) node[below] {$\hat{x}(t)$};

        \draw[->] (osup1.south) -- ($(phi.north) + (-2.4,0cm)$) node[right,pos=0.6] {\small $\hat{\bar{x}}^{1}(t)$};
        \draw[->] (osupo.south) -- ($(phi.north) + (-0.7,0cm)$) node[right,pos=0.6] {\small $\hat{\bar{x}}^{\bar{o}}(t)$};
        \draw[->] (osub1.south) -- ($(phi.north) + (0.7,0cm)$) node[right,pos=0.6] {\small $\hat{\underline{x}}^{1}(t)$};
        \draw[->] (osubo.south) -- ($(phi.north) + (2.4,0cm)$) node[right,pos=0.6] {\small $\hat{\underline{x}}^{\underline{o}}(t)$};       
    \end{tikzpicture}
    \vspace{-1em}
    \caption{Schematic diagram of the secure state observer.}
    \label{fig:observer-diagram}
\end{figure}
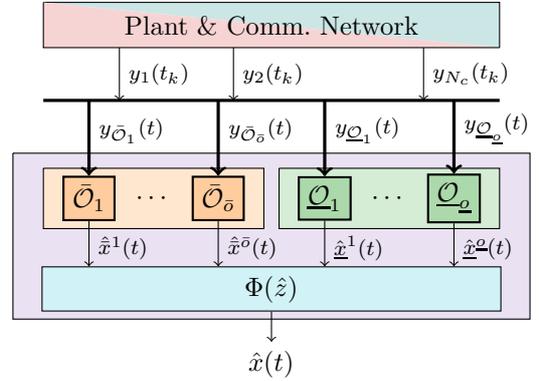

\subsection{Convergence analysis}
We now analyze the convergence of the multi-observer error dynamics. We derive the error dynamics ${\tilde{x}}^S := {x} - {\hat{x}}^S$ of each observer in the multi-observer $S \in \{\supobsv,\subobsv\}$, using \eqref{eq:standard-system-form} and \eqref{eq:multi-observer},
\begin{equation*}\label{eq:dot-tilde-x}
    \begin{split}
        \dot{\tilde{x}}^S = A\dot{\tilde{x}}^S + L^S(C_S\tilde{x}^S(t_k) + w_S(t_k)) + B \phi(Cx + u + d) \\ - B\phi( C\hat{x}^S + u - K^{S}(C_{S}\tilde{x}^S(t_k) + w_{S}(t_k) )), \\
    \end{split}
\end{equation*}
where $w_{S}:= d_S+a_{S}+\eta_{S}$ and we use $C_{S}\hat{x}^S(t_k) + u_{S}(t_k)  - y_{S}(t_k) = -C_S\tilde{x}^S(t_k) - w_{S}(t_k)$. Under Assumption \ref{as:sampling}, the sector condition \eqref{eq:incremental-sector-condition} holds and hence, $\phi(a) - \phi(b) = \varepsilon(t)(a-b), \ \forall a,b \in \mathbb{R}^{N_c}$, where $\varepsilon(t) \in \mathbb{R}^{N_c \times N_c}$ is a diagonal matrix with diagonal components $0 \leq \varepsilon_{ii}(t) \leq \zeta_i,\ i \in [N_c]$. Thus,
\begin{equation}\label{eq:interim-tildex-s}
    \begin{split} 
        \dot{\tilde{x}}^S = &A\tilde{x}^S + B\varepsilon(t)C\tilde{x}^S + (B\varepsilon(t)K^S + L^S) w_S(t_k) \\ & + B\varepsilon(t)d + (B\varepsilon(t)K^S + L^S)C_S \tilde{x}^S(t_k).
    \end{split}
\end{equation}
We analyze the stability of the sampled-data  estimation error system \eqref{eq:interim-tildex-s} under an impulsive system framework developed in \cite{nagh2008exponential}. To this end, we collect all individual state error vectors $\tilde{x}^S$ into one vector $\tilde{z} = (\tilde{x}^S)_{S\in\{\bar{\mathcal{O}}, \underline{\mathcal{O}}\}}$. We aim to show that all the state estimation error systems, i.e., the $\tilde{z}$-system, is input-to-state stable with respect to the sampling error $\delta(t):=\tilde{z}(t)-\tilde{z}(t_k)$, disturbances $d(t)$ and $w(t_k)$. 

\begin{thm}\label{thm:coupled-system}
    Consider system \eqref{eq:standard-system-form}, multi-observer \eqref{eq:multi-observer}, and nonlinearities $\phi_i(\cdot)$ satisfying Assumption \ref{as:nonlinearity}. Let $\mathbf{A} = \mathbb{I}_{N_{\mathcal{O}}} \otimes A$, $\mathbf{B} = \mathbb{I}_{N_{\mathcal{O}}} \otimes B$, $\mathbf{C} = \mathbb{I}_{N_{\mathcal{O}}} \otimes C$,
    \begin{equation*}
        \begin{split}
            \bar{\bfE} &= \mathbb{I}_{N_{\mathcal{O}}} \otimes \operatorname{diag}\left(\zeta_1,\zeta_2,\dots,\zeta_{N_c} \right),  \\
            \mathbf{C}^* &= \operatorname{diag}\left(C_{\supobsv_1},\dots,C_{\supobsv_{\bar{o}}},C_{\subobsv_1},\dots,C_{\subobsv_{\underline{o}}}\right).
        \end{split} 
    \end{equation*}
    Suppose there exist
    \begin{enumerate}[(i)]
        \item scalars $\bar{T} > 0$ and $(\nu,\mu_d,\mu_w) \in \mathbb{R}^3_{>0}$;
        \item symmetric and positive definite matrices $P_1$, $P_2$, $P_3 \in \mathbb{R}^{N_{\mathcal{O}}N_c}$;
        \item not necessarily symmetric matrices $N_1$, $N_2$, $N_3$, $N_4$, $N_5$, $N_6  \in \mathbb{R}^{N_{\mathcal{O}}N_c}$,
        \item a diagonal matrix $U=\operatorname{diag}(u_1,\dots,u_{N_{\mathcal{O}}N_c}) \in \mathbb{R}^{N_{\mathcal{O}}N_c \times N_{\mathcal{O}}N_c}$, where $u_i > 0$, $i\in [N_{\mathcal{O}}N_c]$;
        \item a $\mathbf{K} = \operatorname{diag}\left(\bar{K}^1,\dots,\bar{K}^{\bar{o}},\underline{K}^1,\dots,\underline{K}^{\underline{o}}\right)$;
        \item and an $\mathbf{L} = \operatorname{diag}\left(\bar{L}^1,\dots,\bar{L}^{\bar{o}},\underline{L}^1,\dots,\underline{L}^{\underline{o}}\right)$,
    \end{enumerate}
    such that the matrix inequalities \eqref{eq:simplified-LMI-cond}, \eqref{eq:LMI-cond-1}, and \eqref{eq:LMI-cond-2} are satisfied. Then for any set of sampling times $\mathcal{D}$ satisfying Assumption \ref{as:sampling} with maximum inter-sample time $\bar{T}$, the estimation errors $\tilde{z}$ satisfy
    \begin{equation} \label{eq:sample-data-conclude}
        |\tilde{z}(t)| \leq \max\{ \beta_z(|\tilde{z}(0)|,t),\ \gamma_z(||(\delta(t),d(t),w(t_k))||_{[0,t]}) \},
    \end{equation}
    where $\delta(t) = \tilde{z}(t) - \tilde{z}(t_k)$, $w := d+a+\eta$, $\beta_{z} \in \mathcal{KL}$, and $\gamma_{z} \in \mathcal{K}_\infty$.
\end{thm}
\begin{pf}
    Appendix \ref{ap:coupled-system-proof}
\end{pf}

\begin{table*}[!b] 
    \rule{\textwidth}{0.4pt}
    \centering
    \begin{equation}\label{eq:simplified-LMI-cond}
        \begin{split}
            \begin{bmatrix}
                    P_1\bfA + \bfA^TP_1 + \nu\mathbb{I}_{N_{\mathcal{O}}N_c} & \star  &  \star & \star & \star & \star & \star \\
                    {\bfC^*}^T\bfL^TP_1 & 0 & \star & \star & \star & \star & \star \\
                    \bfB^TP_1 + U\bfC & 0 & -2U\bar{\bfE}^{-1} & \star & \star & \star & \star \\
                    \bfB^TP_1 +  U\bfK\bfCs & 0 & 0 & -2U\bar{\bfE}^{-1} & \star & \star & \star \\
                    P_1 & 0 & 0 & 0 & -\mu_d\mathbb{I}_{N_{\mathcal{O}}N_c} & \star & \star \\
                    P_1 & 0 & 0 & 0 & 0 & -\mu_w\mathbb{I}_{N_{\mathcal{O}}N_c} & \star \\
                    0 & -U\bfK\bfCs & 0 & 0 & 0 & 0& U\bar{\bfE}^{-1}
                \end{bmatrix} \leq 0,
        \end{split}
    \end{equation}
    \begin{equation}\label{eq:LMI-cond-1}
        \begin{bmatrix}
            - P_3 + N_1 + N_1^T & \star & \star & \star & \star & \star & \star & \star \\
            P_3 + N_2 - N_1^T & - P_3 -N_2 - N_2^T & \star & \star & \star & \star & \star & \star \\
            N_3& -N_3 & 0 & \star & \star & \star & \star & \star \\
            N_4 &  - N_4& 0 & 0 & \star & \star & \star & \star \\
            N_5 & - N_5 & 0 & 0 & 0 & \star & \star & \star \\
            N_6 & - N_6 & 0 & 0 & 0 & 0 & \star & \star \\
            \bfA^TP_2 & {\bfC^*}^T\bfL^TP_2 & \bfB^TP_2& \bfB^TP_2 & P_2 & P_2 & -{\bar{T}}^{-1}P_2 & \star \\
            N_1^T & N_2^T & N_3^T & N_4^T & N_5^T & N_6^T & 0  & -{\bar{T}}^{-1}P_2
        \end{bmatrix} < 0,
    \end{equation}
    \begin{equation}\label{eq:LMI-cond-2}
        \begin{bmatrix}
            \alpha_{11}(P_1,P_3,N_1) & \star & \star & \star & \star & \star & \star \\ 
            \alpha_{21}(P_1,P_3,N_1,N_1) & \delta_{22}(P_3,N_2) & \star & \star & \star & \star & \star \\
            \bar{T}\bfB^TP_3 + N_3 & -\bar{T}\bfB^TP_3 -N_3 & 0 & \star & \star & \star & \star \\
            \bar{T}\bfB^TP_3 + N_4 & -\bar{T}\bfB^TP_3 -N_4 & 0 & 0 & \star & \star & \star \\
            \bar{T}P_3 + N_5 & -\bar{T}P_3 -N_5 & 0 & 0 & 0 & \star & \star \\
            \bar{T}P_3 + N_6 & -\bar{T}P_3 -N_6 & 0 & 0 & 0 & 0 & \star \\
            \bfA^TP_2 & {\bfC^*}^T\bfL^TP_2 & \bfB^TP_2 & \bfB^TP_2 & P_2 & P_2 & -{\bar{T}}^{-1}P_2  \\
        \end{bmatrix} < 0,
    \end{equation}
    \begin{equation*}
        \begin{split}
            \alpha_{11}(P_1,P_3,N_1) & =\bar{T}(P_3\bfA + \bfA^TP_3) -P_3 + N_1 + N_1^T, \\
            \alpha_{21}(P_1,P_3,N_1,N_1) & = \bar{T}{\bfC^*}^T\bfL^TP_3 - \bar{T}P\bfA + P_3 + N_2 - N_1^T, \\
            \alpha_{22}(P_3,N_2) & = -\bar{T}(P_3\bfL\bfC^* + {\bfC^*}^T\bfL^TP_3) - P_3 -N_2-N_2^T. \\
        \end{split}
    \end{equation*}
\end{table*}

All that remains is to select the final state estimate $\hat{x}$ from all state estimates $\hat{x}^S$, $S\in\{\bar{\mathcal{O}},\underline{\mathcal{O}}\}$ by applying the selection procedure $\Phi(\hat{z})$, as in \eqref{eq:consistency-measure} and \eqref{eq:selection-criterion}. We will now show that applying this selection procedure results in a bound on $|\tilde{x}(t)|$ that is independent of the attack $a(t)$.

\begin{thm}\label{prop:multi-observer}
    Consider system \eqref{eq:standard-system-form}, multi-observer \eqref{eq:multi-observer}, selection procedure \eqref{eq:consistency-measure} and \eqref{eq:selection-criterion},  nonlinearities $\phi_i(\cdot)$ satisfying Assumption \ref{as:nonlinearity}, an attack $a(t)$ satisfying Assumption \ref{as:attack}, and sensor sampling satisfying Assumption \ref{as:sampling}. Suppose that Theorem \ref{thm:coupled-system} is satisfied, then the state estimation error $\tilde{x}$ satisfies
    \begin{equation}\label{eq:prop-upper-bound}
        |\tilde{x}(t)| \leq \max\{ \beta_x( |\tilde{x}(0)|,t), \gamma_x(\|(\delta(t),d(t),\eta(t_k))\|_{[0,t]})\}
    \end{equation}
    for all $t \geq 0$ and all initial conditions $\tilde{x}(0)\in\mathbb{R}^{N_c}$, where $\beta_x \in \mathcal{KL}$ and $\gamma_x \in \mathcal{K}_\infty$. 
\end{thm}
\begin{pf*}{Sketch of proof.}
    Consider $\delta(t)$ as an additional disturbance, then the proofs of Theorem 2 and Proposition 1 in \cite{Chong2020AAttacks} carry over directly.
\end{pf*}
We have now shown that the proposed state estimator is secure, as defined in Section \ref{sec:problem-statement}.

\section{Case study: Securely monitoring a power distribution network}\label{sec:power-distirbution-network}
We are now ready to demonstrate the capabilities of the proposed secure state estimator on our motivational use case. We first introduce the system -- a low voltage power distribution network, and conclude this section with simulations.

\subsection{Model of a low voltage power distribution network}
Consider an inverter-based, low voltage, power distribution network with $N_c$ customers in a line configuration, as shown in Figure \ref{fig:distribution-network}. Each customer $i\in [N_c]$ is equipped with an inverter denoted by $\Sigma_i$, which is capable of generating both active $\rho_{g,i}$ and reactive $q_{g,i}$ power. Customers also consume both active $\rho_{c,i}$ and reactive $q_{c,i}$ power. The voltage measurements are sampled by a centralized monitoring center periodically, in accordance with Assumption \ref{as:sampling}, making them susceptible to malicious manipulation during transmission.

\begin{figure}[ht]
    \centering
    \includegraphics[width=\linewidth]{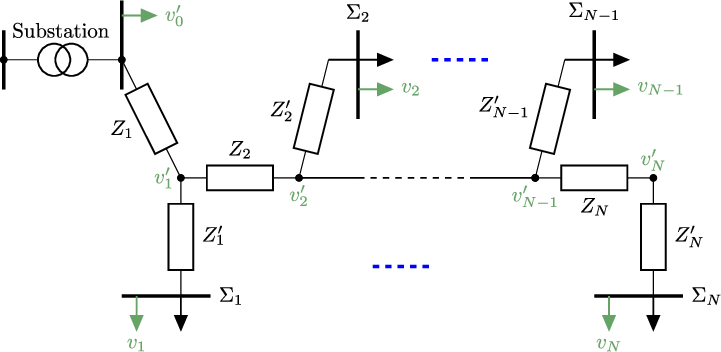}
    \caption{Radially interconnected distribution network}
    \label{fig:distribution-network}
\end{figure}

We model the radial distribution network using the linearized DistFlow model in \cite{network-reconfiguration} and assume that the power losses in the distribution line are negligible. The voltage at the substation is denoted by $v_0'$, the point-to-point voltage of customer connection points on the distribution line are denoted by $v_i'$ and the customer voltages are denoted by $v_i$. The impedances in between each connection point are denoted by $Z_i=R_i + jX_i$ and impedances from connection point to customer are denoted by $Z_i'=R_i' + jX_i'$, where $R_i,R_i'$ are resistances, $X_i,X_i'$ are reactances and $j=\sqrt{-1}$ is the imaginary unit.

Each inverter $\Sigma_i$ is equipped with a droop controller as designed in \cite{droop-controllers}, which regulates customer voltages by generating reactive power $q_{g,i}$ according to $\dot{q}_{g,i} = -a_{g,i}q_{g,i} + a_{g,i}\phi_i(\bar{v}^2-v_i^2)$,
where $a_{g,i} \in \mathbb{R}_{>0}$ is an inverter characteristic, $\bar{v} \in \mathbb{R}$ is a reference voltage level that is communicated to the customer, and $\phi_i: \mathbb{R} \to \mathbb{R}$ is a  saturated dead-zone function
\begin{equation*}
    \phi_i(w)=
    \begin{cases}
        -\bar{q}_i, &\quad w \leq w_{min}, \\
        -\left( 1 - \frac{w-w_{min}}{w_{m}-w_{min}} \right) \bar{q}_i, &\quad w_{min} < w \leq w_{m}, \\
        0, &\quad w_m < w \leq w_n, \\
        \left( \frac{w - w_n}{w_{max} - w_n} \right) \bar{q}_i, &\quad w_n < w \leq w_{max}, \\ 
        \bar{q}_i, &\quad w > w_{max}, \\
    \end{cases}
\end{equation*}
with $\bar{q}_i \in \mathbb{R}_{>0}$ representing the saturation limit of the inverter, satisfying the relation $\bar{q}_i = \sqrt{\bar{s}_i^2 - \rho_{g,i}^2}$. Here, $\bar{s}_i \in \mathbb{R}$ denotes the maximum apparent power of the inverter. Furthermore, note that $\phi_i$ satisfies the incremental sector condition \eqref{eq:incremental-sector-condition} with $\zeta_i=1,\ i \in [N_c]$ and therefore Assumption \ref{as:nonlinearity} is satisfied. 

The same change of coordinates as in \cite{droop-controllers} is applied to transform the system into the form of \eqref{eq:standard-system-form}. The state is selected as a column vector of the reactive powers generated by the customer inverters $\Sigma_i$,
\begin{equation}\label{eq:qg-state}
    x =
    \begin{pmatrix}
        q_{g,1} & q_{g,2} & \cdots & q_{g,N_c}
    \end{pmatrix}.
\end{equation}
We select the auxiliary input $u_i$ in \eqref{eq:standard-system-form} as $u_i = \bar{v}^2 - v_0'^{2} + o_i(\rho,q_c)$,
where the function $o_i:\mathbb{R}^{N_c} \times \mathbb{R}^{N_c} \to \mathbb{R}$ injects the disturbance generated by customers down the line. This choice results in $m_i$ being the squared voltage deviation for each customer $\bar{v}^2 - v_i^2$. Let us define this disturbance function as
\begin{equation*}
    \begin{split}      
        o_i(\rho,q_c) &= \sum^{i}_{j=1}{\bar{o}_j(\rho,q_c)} + \sum^{i-1}_{j=1}{\beta_{j}'(\rho_{j},q_{c,j})}, \\
    \end{split}
\end{equation*}
where $\bar{o}_j: \mathbb{R}^{N_c} \times \mathbb{R}^{N_c} \to \mathbb{R}$ encapsulates the effect of downstream customers and $\beta'_j: \mathbb{R} \times \mathbb{R} \to \mathbb{R}$ captures the voltage drop. These disturbance functions are defined as
\begin{equation*}
    \bar{o}_j(\rho,q_c) = 2 \sum^{N_c}_{k=j+1}{(X_jq_{c,k} - R_j\rho_k)} - 2\beta_{j}'(\rho_{j+1},q_{c,j+1}),
\end{equation*}
and
\begin{equation*}
    \beta_j'(\rho,q_c) = 
    \begin{cases}
        R_i'\rho + X_i'q_c, & j \geq 0, \\
        0, & j = -1.
    \end{cases}
\end{equation*}
We now construct the system matrices $A,B$ and $C$ as in \eqref{eq:standard-system-form} as follows
\begin{equation}
    A = \operatorname{diag}(-a_{g,1},-a_{g,2},\dots,-a_{g,N}), \quad B = -A
    \label{eq:A-B-matrices}
\end{equation}
and
\begin{equation}
    \begin{split}
        C = -2
        \begin{bmatrix}
            X_1 & X_1 & \cdots & X_1 \\
            X_1 & X_1 + X_2 & \cdots & X_1 + X_2 \\
            \vdots & \vdots & \ddots & \vdots \\
            X_1 & X_1 + X_2 & \dots & \sum\nolimits^{N_c}_{i=1} X_i \\
        \end{bmatrix} \\
        - 2 \operatorname{diag}(X_0',X_1',\dots,X_N').
    \end{split}
    \label{eq:C-matrix}
\end{equation}
From the state ${x}$ in in \eqref{eq:qg-state} which consists of the generated reactive power at the inverters $\Sigma_i$, we can now calculate the voltages at the customers as
\begin{equation}
    {v}_i^2(t) = C_i{x}(t) - {o}_i(\rho(t),q_c(t)) + v_0'^2(t), \quad t \geq 0.
    \label{eq:voltage-calculation}
\end{equation}

\subsection{Securely estimating the state of a power distribution network}
We aim to provide a secure state estimate of the customer voltages as in \eqref{eq:voltage-calculation}. To that end, we apply the proposed secure state estimator developed in Section \ref{sec:observer-design} to generate a state estimate $\hat{x}$ of the true state $x$. We then use this state estimate to compute a voltage estimate given by
\begin{equation}
    \hat{v}_i^2(t) = C_i\hat{x}(t) - {o}_i(\rho(t),q_c(t)) + v_0'^2(t), \quad t \geq 0.
    \label{eq:voltage-estimate}
\end{equation}

Simulation was conducted on a benchmark residential European low-voltage distribution network with $N_c=5$ customers, along with inverter power consumption and generation as provided in Table \ref{tab:simulation-parameters}. The nominal voltage at the substation is $v_0(t) = 230 + \sin(5t)\ (V)$ and the reference voltage communicated to each customer is $\bar{v}=230\ (V)$. 

\begin{table}[h]
    \centering
    \caption{Simulation parameters as in \cite{Strunz-Benchmark-Systems} Figure 7.7 and Table 7.26}
    \begin{tabular}{|c|c|c|c|c|c|}
        \hline
        $i$ & 1 & 2 & 3 & 4 & 5 \\
        \hline
        $R_{i}\ (\Omega)$ & 0.00343 & 0.00172 & 0.00343 & 0.00515 & 0.00172 \\
        \hline
        $X_{i}\ (\Omega)$ & 0.04711 & 0.02356 & 0.04711 & 0.07067 & 0.02356 \\
        \hline
        $R'_{i}\ (\Omega)$ & 0.00147 & 0.00662 & 0.00147 & 0.00147 & 0.00147 \\
        \hline
        $X'_{i}\ (\Omega)$ & 0.02157 & 0.09707 & 0.02157 & 0.02157 & 0.02157 \\
        \hline
        $\rho_{g,i}\ (W)$ & 3500 & 5500 & 4000 & 4500 & 3000 \\
        \hline
        $\rho_{c,i}\ (W)$ & 2295 & 5440 & 5440 & 2295 & 2720 \\
        \hline
        $q_{c,i}\ (VAr)$ & 300 & 960 & 480 & 600 & 400 \\
        \hline
        $\bar{s}_i\ (VA)$ & 4200 & 6500 & 4700 & 5300 & 3600 \\
        \hline
    \end{tabular}
    \label{tab:simulation-parameters}
\end{table}

Using the values for $\bar{s}_i$ and $q_{g,i}$ in Table \ref{tab:simulation-parameters}, we can recover the saturation limits $\bar{q}_i$ according to $\bar{q_i} = \sqrt{\bar{s}_i^2 - \rho_{g,i}^2}$:
\begin{equation*}
    \bar{q} = (2321.6,\  3464.1,\  2467.8,\ 2800.0,\ 1900.0).
\end{equation*}
We select $w_{n,i}=w_{m,i}=0\ (V^2)$ and $w_{max,i}=-w_{min,i}=14899.4\ (V^2)$ for all inverters $i\in[N]$. The state matrices $A,B$ and $C$ are given by \eqref{eq:A-B-matrices} and \eqref{eq:C-matrix}. We find $K^S$ and $L^S$ as follows, we first solve \eqref{eq:simplified-LMI-cond} for $\nu$, $\mu_d$, $\mu_w$, $U$, and $P_1$ and we find some $K^S$ and $L^S$ that satisfy the condition. Using these values we then solve \eqref{eq:LMI-cond-1} and \eqref{eq:LMI-cond-2} for the remaining variables, $P_2$, $P_3$, $N_i,\ i \in [6]$ for a fixed $\bar{T}=1\ (s)$. The solution to the matrix inequalities are computed using YALMIP, see \cite{Lofberg2004}, yielding $\nu=0.261$, $\mu_d=1.260$ and $\mu_w=1.260$.\footnote{The other parameters can be found at https://gitlab.tue.nl/20195429/sse-with-sampled-measurements-for-power-distribution-systems/-/tree/6498c033d621ecc58f63e934f54e455a358e904a/}

We attack the power distribution network with $a_2(t) = -5000 \operatorname{sign}(\sin(t))$, $a_5(t) = 7500\cos(5t)$, and $a_{\{2,3,5\}}(t) = 0$ for all $t \geq 0$. Note that this attack satisfies all conditions as in Assumption \ref{as:attack}. We sample the outputs as follows, $t_0=0$ and the inter-sample times are the elements of the sequence \( \mathcal{T}= \{\bT,\ 0.7\bT,\ 0.2\bT,\ 0.6\bT,\ 0.4\bT,\ \bT,\ 0.9\bT,\ 0.5\bT\}\) repeated indefinitely. Explicitly, the sample times are $t_k = \sum_{i=1}^{k} \mathcal{T}_i$, such that $\mathcal{D}= \{ \sum_{i=1}^{k} \mathcal{T}_i \ : \ k \in \mathbb{N} \}$. Note that this sequence of sampling times satisfies Assumption \ref{as:sampling}. Under this scenario, the state estimation errors $\tilde{x}$ for each customer can be found in Figure \ref{fig:state-estimation-error}. 
\begin{figure}[h]
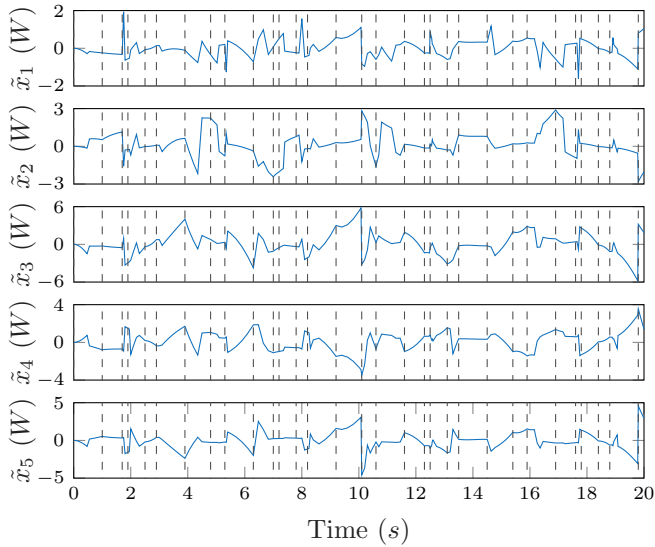

    \centering
    \include{Figures/errorplot}
    \caption{Blue lines indicate the state estimation error $\tilde{x}_i=x_i-\hat{x}_i$ at each customer $i \in [5]$ for $t \in [0,20]$. The grey, dashed lines indicate sampling times $t_k,\ k \in \mathbb{N}$.}
    \label{fig:state-estimation-error}
\end{figure}

It can be seen in Figure \ref{fig:state-estimation-error} that the state estimation error signals $\tilde{x}(t)$ are independent of the attack signal $a(t)$, consistent with Theorem \ref{prop:multi-observer}. The estimated voltage is calculated using \eqref{eq:voltage-calculation} which has an RMS error of $0.0234\ (V)$ in the aforementioned scenario. 

\section{Conclusion}
We have proposed an observer design that can provide a secure state estimate of a nonuniformly sampled, Lur'e-type plant with sector bounded nonlinearites, when less than half of the sensors are being attacked. The convergence of the state estmation error system was analyzed under an impulsive system framework and we demonstrated its efficacy on a low-voltage power distribution network in simulation. Future work will focus on extending the observer design to accomodate measurements which are also asynchronously sampled measurements.

\bibliography{ifacconf}             
                                                   







\appendix
\section{Proof of Theorem \ref{thm:coupled-system}}\label{ap:coupled-system-proof}
We will analyze the sampled-data estimation error system in \eqref{eq:interim-tildex-s} under the impulsive systems framework of \cite{nagh2008exponential}. To this end, we now define ${w}_{S}'(t) = w_{S}(t_k)$ and $e_{\tilde{x}}^S(t) = \tilde{x}^S(t_k)$, as the values of $w_S$ and $\tilde{x}^S$ at the last sampling time $t_k$. Therefore, they have derivatives $\dot{\bar{w}}_S=\dot{e}^S_{\tilde{x}}=0$ when $t\in \mathcal{C}$. Each sampled-data estimation error system can then be viewed as an impulsive system as follows for $t\in\mathcal{C}$,
\begin{equation}
    \begin{split}
        \dot{\tilde{x}}^S =  & A\tilde{x}^S + B\varepsilon(t)C\tilde{x}^S  + (B\varepsilon(t)K^S + L^S) {w}'_S \\ &
        + (B\varepsilon(t)K^S + L^S)C_Se_{\tilde{x}}^S + B\varepsilon(t)d, 
    \end{split}
\end{equation}
and for $t\in\mathcal{D}$, we have $\tilde{x}^{S}(t_k^+)=\tilde{x}^{S}(t_k)$ and $e_{\tilde{x}}^{S}(t_k^+)=e_{\tilde{x}}^{S}(t_k)$.

We now introduce the following for notational compactness: Let ${e}_{\tilde{x}}$ be the vector containing all $\etx^S$, that is $\etx := (\bar{e}_{\tilde{x}}^1,\dots,\bar{e}_{\tilde{x}}^{\bar{o}},\underline{e}_{\tilde{x}}^1,\dots,\underline{e}_{\tilde{x}}^{\underline{o}})$. We also define $\bar{d} := \mathbf{1}_{N_{\mathcal{O}}} \otimes d$, $\bar{w}' := ({w}'_{\supobsv_1},\dots,{w}'_{\supobsv_{\bar{o}}}, {w}'_{\subobsv_1},\dots,{w}'_{\subobsv_{\underline{o}}})$, $\xi := (\tilde{z}, e_{\tilde{x}})$, and let
\begin{equation*}
    \mathcal{X} := (
        \xi,\
        \mathbf{E}(t)\mathbf{C}\tilde{z},\
        \bfE(t)\bfK\bfC^*e_{\tilde{x}},\
        \bfB\bfE(t) \bar{d},\
        (\bfB\bfE(t)\bfK + \bfL)\bar{w}').
\end{equation*}
Note that $\tilde{z},\ e_{\tilde{x}},\ \bar{d},\ \bar{w}' \in \mathbb{R}^{N_{\mathcal{O}}N_C}$, $\xi \in \mathbb{R}^{2N_{\mathcal{O}}N_c}$, and that stacking these vectors results in $\mcX \in \mathbb{R}^{6N_{\mathcal{O}}N_C}$. We now write the dynamics for the ${\tilde{z}}$-system compactly below
\begin{equation}\label{eq:dot-tilde-z}
    \begin{split}
        \dot{\tilde{z}} = \mathbf{A}\tilde{z} + \mathbf{B}\mathbf{E}(t)\mathbf{C} \tilde{z} +\mathbf{BE}(t)\bar{d} + (\mathbf{BE}(t)\mathbf{K} + \mathbf{L})\bar{w}' \\ + \mathbf{BE}(t)\mathbf{K}\bfC^*e_{\tilde{x}} + \mathbf{L}\mathbf{C}^*e_{\tilde{x}}.
    \end{split}
\end{equation}
System \eqref{eq:dot-tilde-z} can also be expressed in terms of $\mcX$, given by $\dot{\tilde{z}} = M(\bfL)\mcX$, where $M(\bfL) = \begin{bmatrix}\bfA & \bfL\bfC^* & \bfB & \bfB & \mathbb{I}_{N_{\mathcal{O}}N_c} & \mathbb{I}_{N_{\mathcal{O}}N_c} \end{bmatrix}$. Therefore, the error system we will analyze can be compactly expressed as
\begin{equation}\label{eq:formal-proof-system}
    \begin{cases}
        \dot{\xi}(t) = (M(\bfL)\mcX(t),\ 0), & t \in \mathcal{C}, \\
        \xi(t_k) = (\tilde{z}(t),\ \tilde{z}(t)), & t \in \mathcal{D},
    \end{cases}
\end{equation}
which is an impulsive system as defined in \cite{nagh2008exponential}.

We now define $\tau = t-t_k$ as the time passed since the last sensor update and consider the following candidate Lyapunov function inspired by \cite{nagh2008exponential} and \cite{raff2008observer} below.
\begin{equation} \label{eq:candidate-lyap-U}
    \begin{split}
        U(\xi,\tau) = \underbrace{\tilde{z}^T P_1 \tilde{z}}_{V_1(\tilde{z})} + \underbrace{\int_{t-\tau}^t (\bar{T}- t +s) \dot{\tilde{z}}^T(s) P_2 \dot{\tilde{z}}(s)ds}_{V_2(\xi,\tau)}\\
        +\underbrace{(\bar{T} - \tau)(\tilde{z} - e_{\tilde{x}})^T P_3 (\tilde{z} - e_{\tilde{x}})}_{V_3(\xi,\tau)},
    \end{split} 
\end{equation}
where $P_1$, $P_2$, and $P_3$ are symmetric, positive definite matrices as previously defined in Theorem \ref{thm:coupled-system}. We now show that $U(\xi,\tau)$ is a valid Lyapunov function according to Theorem 1 in \cite{nagh2008exponential}. For the constants $\underline{a}_U = \lambda_{min}(P_1)$ and \(\bar{a}_U=\max\{ \lambda_{max}(P_1),\ \bar{T}^2\lambda_{max}(P_2),\\ \bar{T}\lambda_{max}(P_3)\}\) it holds that
\begin{equation}
    \underline{a}_U |\xi|^2 \leq U(\xi,\tau) \leq \bar{a}_U |\xi|^2,
\end{equation}
which shows that $U(\xi,\tau)$ is positive definite. Next, we analyze $U(\xi,\tau)$ at sampling times $t\in \mathcal{D}$: $V_1$ remains unchanged at sampling times since $\tilde{z}$ is not updated when the outputs are sampled. The term $V_2$ is zero right after sampling, since the integral has zero length. Finally, $V_3$ is also zero since $e_{\tilde{x}}(t_k) = \tilde{z}(t_k)$ as defined in \eqref{eq:formal-proof-system}. Therefore,
\begin{equation}\label{eq:jumps}
    U(\xi(t_k)),0) \leq \lim_{h \to t_k} U(\xi(h),\tau(h)).
\end{equation}
We now study the behaviour of the candidate Lyapunov function \eqref{eq:candidate-lyap-U} during flows, i.e., $t\in\mathcal{C}$. We start with $V_1$ whose time derivative along the solutions of \eqref{eq:dot-tilde-z} is
\begin{equation}\label{eq:Q1}
    \begin{split}
        \langle &\nabla{V}_1(\xi,\tau ), (\dot{\xi},1) \rangle = \dot{\tilde{z}}^TP_1\tilde{z} + \tilde{z}^TP_1\dot{\tilde{z}} \\
        &= \mcX^T
        \left(
        M^T(\bfL) P_1 
        H_{\tz} + H_{\tz}^TP_1M(\bfL)
        \right) \mcX
        \\
        &=\mcX^T
        \underbrace{
        \begin{bmatrix}
            P_1\bfA + \bfA^TP_1  & P_1\bfL\bfC^*  &  P_1\bfB & P_1\bfB & P_1 & P_1 \\
            {\bfC^*}^T\bfL^TP_1 & 0 & 0 & 0 & 0 & 0 \\
            \bfB^TP_1 & 0 & 0 & 0 & 0 & 0 \\
            \bfB^TP_1 & 0 & 0 & 0 & 0 & 0 \\
            P_1 & 0 & 0 & 0 & 0 & 0 \\
            P_1 & 0 & 0 & 0 & 0 & 0 \\
        \end{bmatrix}}_{Q_1(P_1,P_1\bfL)}\mcX,
    \end{split}
\end{equation}
where $H_{\tz}= \begin{bmatrix}\mathbb{I}_{N_{\mathcal{O}}N_c}&0&0&0&0&0\end{bmatrix}$ and thus $\tz=H_{\tz}\mcX$. Let $\mathcal{G}(U\bfK) = \bfCs^T(U\bfK)^TU^{-1}\bar{\bfE}U\bfK\bfCs$. Since the matrix inequality \eqref{eq:simplified-LMI-cond} holds, we get
\begin{equation}\label{eq:Q2}
    \begin{split}
        &Q_1(P_1,P_1\bfL) \leq \\
        &\underbrace{
        \begin{bmatrix}
            -\nu \mathbb{I}_{N_{\mathcal{O}}N_c} & 0 & -\bfC^T U & -{\bfC^*}^T(U\bfK)^T & 0 & 0 \\
            0 & -\mathcal{G}(U\bfK) & 0 & 0 & 0 & 0 \\
            -U\bfC & 0 & 2U\bar{\bfE}^{-1} & 0 & 0 & 0 \\
            -U\bfK\bfC^* & 0 & 0 & 2U\bar{\bfE}^{-1} & 0 & 0 \\
            0 & 0 & 0 & 0 & \mu_d \mathbb{I} & 0 \\
            0 & 0 & 0 & 0 & 0 & \mu_w \mathbb{I}
        \end{bmatrix}}_{Q_2(U,U\bfK,\nu, \mu_d,\mu_w)},
    \end{split}
\end{equation}
which we derive by applying the Schur complement to \eqref{eq:simplified-LMI-cond}. Therefore,
\begin{equation}\label{eq:V1-bound-w-Q2}
    \begin{split}
        \langle \nabla{V}_1(\xi,\tau ), &(\dot{\xi},1) \rangle \leq  -\nu |\tilde{z}|^2  \\
        &-2 \tilde{z}^T ({\bfC}^T U \bfE(t) \bfC  -  {\bfC}^T \bfE(t) U \bar{\bfE}^{-1} \bfE(t) \bfC) \tilde{z} \\
        &- e_{\tilde{x}}^T {\bfC^*}^T \bfK^T (U \bar{\bfE} - 2 \bfE(t) U \bar{\bfE}^{-1} \bfE(t) )\bfK \bfC^* e_{\tilde{x}} \\
        &-2\tilde{z}^T {\bfC^*}^T\bfK^T U \bfE(t) \bfK \bfC^* e_{\tilde{x}} \\ &+ \mu_d | \bfB\bfE(t)\bd|^2 + \mu_w | (\bfB\bfE(t)\bfK + \bfL)\bar{w}'|^2.
    \end{split}
\end{equation}
Recall that $\delta(t) = \tilde{z}(t) - e_{\tilde{x}}(t)$, substituting this in \eqref{eq:V1-bound-w-Q2} leads to
\begin{equation*}
    \begin{split}
        \langle \nabla{V}_1(\xi,\tau )&, (\dot{\xi},1) \rangle \leq  -\nu |\tilde{z}|^2 \\ &- e_{\tilde{x}}^T{\bfC^*}^T \bfK^T U \bar{\bfE} \bfK \bfC^* e_{\tilde{x}} \\
        &-2 \tilde{z}^T {\bfC}^T (U \bfE(t) - \bfE(t) U \bar{\bfE}^{-1} \bfE(t)) \bfC \tilde{z} \\
        &- 2e_{\tilde{x}}^T {\bfC^*}^T \bfK^T( U {\bfE}(t) -  \bfE(t) U \bar{\bfE}^{-1} \bfE(t)) \bfK \bfC^* e_{\tilde{x}} \\
        & - 2 \delta^T {\bfC^*}^T\bfK^T U \bfE(t) \bfK \bfC^* e_{\tilde{x}} \\ &+ \mu_d | \bfB\bfE(t)\bd|^2 + \mu_w | (\bfB\bfE(t)\bfK + \bfL)\bar{w}'|^2.
    \end{split}
\end{equation*}
We now analyze the fourth and fifth terms component wise, where we see that
\begin{equation*}
    u_i \varepsilon_i(t) - \frac{u_i\varepsilon_i^2(t)}{\zeta_i} = u_i\varepsilon_i(t) \left( 1 - \frac{\varepsilon_i(t)}{\zeta_i} \right) \geq 0,
\end{equation*}
since $\varepsilon_i(t) \in [0,\zeta_i]$, according to Assumption \ref{as:nonlinearity}. Hence,
\begin{equation}\label{eq:V1-ineq-w-cross-term}
    \begin{split}
        \langle \nabla&{V}_1(\xi,\tau ), (\dot{\xi},1) \rangle \leq  
        -\nu |\tilde{z}|^2 \\ &- \mu_d|\bfB\bfE(t)\bar{d}|^2  + \mu_w |(\bfB\bfE(t)\bfK + \bfL)\bar{w}'|^2\\ 
        &-2e_{\tilde{x}}^T{\bfC^*}^T \bfK^T U \bar{\bfE} \bfK \bfC^* e_{\tilde{x}} -2 \delta^T {\bfC^*}^T\bfK^T U \bfE(t) \bfK \bfC^* e_{\tilde{x}}. \\
    \end{split}
\end{equation}
For any $a,b \in \mathbb{R}^n$ and diagonal, positive definite matrix $\Omega=\Delta^T\Delta \in \mathbb{R}^{n \times n}$,
\begin{equation}\label{eq:omega-ineq}
    \begin{split}
       -2a^T\Omega b =- 2 (\Delta a)^T(\Delta b) \leq (\Delta a)^T(\Delta a) + (\Delta b)^T (\Delta b) \\= a^T \Omega a+ b^T \Omega b,
    \end{split}
\end{equation}
which can be derived from $(a+b)^T\Omega(a+b) \geq 0$. Therefore, by applying \eqref{eq:omega-ineq} on the last term in \eqref{eq:V1-ineq-w-cross-term} we can simplify the upper bound to
\begin{equation*}
    \begin{split}
        \langle \nabla{V}_1(\xi,\tau )&, (\dot{\xi},1) \rangle \leq \delta^T {\bfC^*}^T\bfK^T U \bfE(t) \bfK \bfC^* \delta  
        -\nu |\tilde{z}|^2 \\  &- \mu_d|\bfB\bfE(t)\bar{d}|^2  + \mu_w |(\bfB\bfE(t)\bfK + \bfL)\bar{w}'|^2\\ 
        &-e_{\tilde{x}}^T{\bfC^*}^T \bfK^T U( \bar{\bfE} - \bfE(t)) \bfK \bfC^* e_{\tilde{x}}, \\
    \end{split}
\end{equation*}
where, since $\zeta_i\geq\varepsilon_i(t) \implies \bar{\bfE} - \bfE(t) \geq 0$, the last term is always negative and can be omitted. Hence, we arrive at
\begin{equation} \label{eq:V1dot}
    \begin{split}
        \langle \nabla{V}_1(\xi,\tau ), &(\dot{\xi},1) \rangle \leq -\nu |\tilde{z}|^2 + |U \bfE(t)|| \bfK \bfC^*|^2 |\delta|^2  \\
        &+ \mu_d|\bfB\bfE(t)|^2|\bar{d}|^2  + \mu_w |(\bfB\bfE(t)\bfK + \bfL)|^2|\bar{w}'|^2. \\
    \end{split}
\end{equation}
We now analyze $V_2$: its time derivative along the solutions to \eqref{eq:formal-proof-system} is
\begin{equation}\label{eq:V2-derivative}
    \begin{split}
        \langle \nabla V_2(\xi,\tau), (\dot{\xi},1)\rangle = \bar{T} \dot{\tilde{z}}^T P_2 \dot{\tilde{z}} - \int_{t-\tau}^t \dot{\tilde{z}}^T(s) P_2 \dot{\tilde{z}}(s)ds, \\
    \end{split}
\end{equation}
where the first term can be written as
\begin{equation}\label{eq:Q3}
    \begin{split}
    \bar{T}\dot{\tilde{z}}^TP_2\dot{\tilde{z}} = \bar{T}\mcX^T
    \underbrace{
    M^T(\bfL)P_2P_2^{-1}P_2M(\bfL)}_{Q_3(P_2,P_2\bfL)}\mcX.
    \end{split}
\end{equation}
Now consider 
\begin{equation}\label{eq:integral-trick}
    0 \leq \int_{t-\tau}^t
    \begin{bmatrix}
        \dot{\tilde{z}}(s) \\
        \mcX(t)
    \end{bmatrix}^T
    \underbrace{
    \begin{bmatrix}
        P_2 & N^T \\
        N & NP_2^{-1}N^T
    \end{bmatrix}}_{\bar{N}}
    \begin{bmatrix}
        \dot{\tilde{z}}(s) \\
        \mcX(t)
    \end{bmatrix}
    ds,
\end{equation}
which holds for any $N = [N_1^T\ N_2^T\ N_3^T\ N_4^T\ N_5^T\ N_6^T]^T$, with $N_i \in \mathbb{R}^{N_{\mathcal{O}}N_C \times N_{\mathcal{O}}N_C},\ i \in [6]$. Since
\begin{equation*}
    \bar{N}=
    \begin{bmatrix}
        P_2 & N^T \\
        N & NP_2^{-1}N^T
    \end{bmatrix}
    =
    \begin{bmatrix}
        \mathbb{I} \\
        NP_2^{-1}
    \end{bmatrix}
    P_2
    \begin{bmatrix}
        \mathbb{I} & P_2^{-1}N^T
    \end{bmatrix},
\end{equation*}
is positive semidefinite since $P_2$ is positive definite and symmetric. 
Applying \eqref{eq:integral-trick} to  \eqref{eq:V2-derivative} gives
\begin{equation}\label{eq:V2dot}
    \begin{split}
        \langle \nabla V_2(\xi,\tau), (\dot{\xi},1)\rangle \leq \mcX^T (\bar{T}Q_3 + \tau(t) Q_4) \mcX \hspace{7em} \\
        + \mcX^T
        \underbrace{
        \begin{bmatrix}
            N_1 + N_1^T &  - N_1+N_2^T & N_3^T & N_4^T & N_5^T & N_6^T \\
            N_2 - N_1^T & -N_2-N_2^T & - N_3^T & - N_4^T & -N_5^T & -N_6^T \\
            N_3 & -N_3 & 0 & 0 & 0 & 0 \\
            N_4 & -N_4 & 0 & 0 & 0 & 0 \\
            N_5 & -N_5 & 0 & 0 & 0 & 0 \\
            N_6 & -N_6 & 0 & 0 & 0 & 0 \\
        \end{bmatrix}}_{Q_5(N)}\mcX,
    \end{split}
\end{equation}
where $Q_4(NP_2):=NP_2^{-1}N^T$ and
$Q_5(N):=N(H_{\tz} - H_{\etx}) + (H_{\tz}^T - H_{\etx}^T)N^T$ with $H_{\etx} = \begin{bmatrix}0&\mathbb{I}_{N_{\mathcal{O}}N_c}&0&0&0&0\end{bmatrix}$ and $H_{\tilde{z}} = \begin{bmatrix}\mathbb{I}_{N_{\mathcal{O}}N_c}&0&0&0&0&0\end{bmatrix}$, respectively, is derived using $\etx := H_{\etx}\mcX$ and $\tilde{z}:= H_{\tilde{z}} \mcX$, as
\begin{equation*}
    \begin{split}
        \mcX^T&N(\tilde{z}-e_{\tilde{x}}) + (\tilde{z} - e_{\tilde{x}})^TN^T\mcX = \\ &\mcX^T(N(H_{\tz} - H_{\etx}) + (H_{\tz}^T - H_{\etx}^T)N^T)\mcX = \mcX^T Q_5(N) \mcX.
    \end{split}
\end{equation*}
Finally $V_3$ is analyzed, which has the time derivative along the solutions to \eqref{eq:formal-proof-system} as follows
\begin{equation} \label{eq:V3dot}
    \begin{split}
        \langle \nabla  V_3 (&\xi,\tau), (\dot{\xi},1)\rangle \\  = & -(\tilde{z} - e_{\tilde{x}})^T P_3 (\tilde{z} - e_{\tilde{x}})\\& \quad + (\bar{T} - \tau(t))\left((\tilde{z} - e_{\tilde{x}})^T P_3 \dot{\tilde{z}} + \dot{\tilde{z}}^T P_3 (\tilde{z} - e_{\tilde{x}})  \right)\\
        = & \mcX\T(Q_6 + (\bT-\tau(t))Q_7)\mcX,
    \end{split}
\end{equation}
where
\begin{equation}\label{eq:Q6}
    Q_6(P_3) := -(H_{\tz}^T - H_{\etx}^T)P_3(H_{\tz} - H_{\etx})
\end{equation}
and
\begin{equation}\label{eq:Q7}
    \begin{split}
        &Q_7(P_3,P_3\bfL) \\
        &:= (H_{\tz}^T - H_{\etx}^T)P_3M(\bfL) + M(\bfL)^TP_3(H_{\tz} - H_{\etx}) \\
        &=
        \begin{bmatrix}
            P_3\bfA + \bfA^TP_3 & W - \bfA^TP_3 & P_3\bfB & P_3\bfB & P_3 & P_3 \\
            -P_3\bfA + W^T & -W - W^T & -P_3\bfB & -P_3\bfB & -P_3 & -P_3 \\
            \bfB^TP_3 & -\bfB^TP_3 & 0 & 0 & 0 & 0 \\ 
            \bfB^TP_3 & -\bfB^TP_3 & 0 & 0 & 0 & 0 \\ 
            P_3 & -P_3 & 0 & 0 & 0 & 0 \\ 
            P_3 & -P_3 & 0 & 0 & 0 & 0 \\
        \end{bmatrix},
    \end{split}
\end{equation}
with $W := P_3\bfL\bfC^*$. We combine the three bounds on the time derivatives \eqref{eq:V1dot}, \eqref{eq:V2dot} and \eqref{eq:V3dot} to bound the full derivative as follows
\begin{equation*}\label{eq:Vdot-bounds}
    \dot{U} \leq \mcX^T(Q_1 - Q_2)\mcX + \mcX^T(\underbrace{R_1 + \tau R_2 + (\bar{T} - \tau)R_3}_{\bar{Q}(\tau)})\mcX,
\end{equation*}
where $R_1 := \bT Q_3 + Q_5 + Q_6$, $R_2 := Q_4$, and $R_3 := Q_7$, with the $Q_i$ matrices as defined in \eqref{eq:Q1}, \eqref{eq:Q2}, \eqref{eq:Q3}, \eqref{eq:V2dot},  \eqref{eq:Q6}, and \eqref{eq:Q7}. Consider \eqref{eq:LMI-cond-1} and \eqref{eq:LMI-cond-2}, we apply the Schur complement to both inequalities to derive
\begin{equation}\label{eq:seperate-conditions}
    R_1 + \bar{T}R_2 < 0, \quad R_1 + \bar{T}R_3 < 0.
\end{equation}
We now show that these matrix inequalities imply that $\bar{Q}(\tau) < 0$. To that end, consider a $\psi \in [0,1]$ and note that \eqref{eq:seperate-conditions} implies that
\begin{equation*}
    \begin{split}
    \psi(R_1 + \bar{T}R_2)+(1-\psi)&(R_1+\bar{T}R_3) = \\ &R_1 + \psi\bT R_2 + (\bT - \psi\bT)R_3 < 0. 
    \end{split}
\end{equation*}
Hence, it should also hold for $\psi = \tau/\bT$ and therefore $\bar{Q}(\tau) < 0$ for any $\tau \in [0,\bT]$. We can now conclude that
\begin{equation*}
        \begin{split}
        \dot{U} \leq &-\nu |\tilde{z}|^2 - \kappa|\mcX|^2, + |U \bfE(t)|| \bfK \bfC^*|^2 |\delta|^2 \\ &+ \mu_d|\bfB\bfE(t)|^2|\bar{d}|^2
        + \mu_w |(\bfB\bfE(t)\bfK + \bfL)|^2|\bar{w}'|^2, \\
    \end{split}
\end{equation*}
where $\kappa := -\max_{\tau \in [0,\bT]} \lambda_{\max}\left( \bar{Q}(\tau) \right) > 0$. Therefore,
$\dot{U} \leq -\Pi|\xi|^2 + \Theta|(\delta,\bar{d},\bwp)|^2$, where
\begin{equation*}
    \begin{split}
        \Pi &:= \min\{ \nu + \kappa + \kappa|\bfEt\bfC|^2, \kappa + \kappa |\bfEt\bfK\bfCs|^2 \} \\
        \Theta &:= \max \{|U\bfEt||\bfK\bfCs|^2, \mu_d|\bfB\bfEt|^2,\\ &\mspace{192.5mu} \mu_w|\bfB\bfEt\bfK+\bfL|^2 \} \\
    \end{split}
\end{equation*}
 Note that we split up $|\mcX|^2$ component by component here and omitt all negative terms. Therefore, $\dot{U}(\xi,\tau) \leq -\frac{\Pi}{\bar{a}_U}U(\xi,\tau) + \Theta|(\delta,\bar{d},\bar{w}')|^2$. Since $|\tilde{z}|^2 \leq |\xi|^2$, $|\bd|= N_{\mathcal{O}}|d|$, $|\bwp| \leq N_{\mathcal{O}}|w'|$, and most importantly \eqref{eq:jumps}, we obtain $$|\xi(t)| \leq \max\{{\beta}_\xi(|\xi(0)|,t),\ \gamma_\xi(||(\delta,d,w')||_{[0,t]})\},$$
where
\begin{equation*}
    {\beta}_\xi(r,t) := 2r\sqrt{\frac{\bar{a}_U}{\underline{a}_U}}e^{-\frac{\Pi}{2 \bar{a}_{U}}t}, \quad \gamma_\xi(r) := 2N_{\mathcal{O}}\sqrt{\frac{\Theta} {\underline{a}_U}}r,
\end{equation*}
for $t \in \mathbb{R}_{\geq 0}$. Under Assumption \ref{as:sampling}, we have $t_0=0$ which implies $|\xi(0)|=2|\tz(0)|$ and we obtain \eqref{eq:sample-data-conclude} with $\beta_z := 2{\beta}_\xi$ and $\gamma_z := \gamma_\xi$.

\end{document}